\documentclass[10pt]{amsart}
\usepackage{amsmath}
\usepackage[usenames,dvipsnames]{color}
\usepackage{parskip}
\usepackage{amsfonts}
\usepackage{amscd}
\usepackage[centertags]{amsmath}
\usepackage{amssymb}
\usepackage[all,cmtip]{xy}
\usepackage[english]{babel}
\usepackage{tabularx}
\usepackage{mathtools}
\usepackage{amsxtra}
\usepackage{euscript}
\usepackage[T1]{fontenc}
\usepackage{doc, exscale, fontenc, latexsym, syntonly}
\usepackage{amsfonts}
\usepackage{amsthm}
\usepackage{graphicx}
\usepackage{xcolor}
\usepackage{tikz-cd}

\numberwithin{figure}{section}
\numberwithin{table}{section}

\newcommand{\scat}{\ensuremath{\mathrm{scat}}}

\newcommand{\id}{\ensuremath{\mathrm{1}}}
\newcommand{\TC}{\ensuremath{\mathrm{TC}}}

\newcommand{\pr}{\ensuremath{\mathrm{pr}}}

\newtheorem{theorem}{Theorem}[section]
\newtheorem{definition}{Definition}[section]
\newtheorem{corollary}{Corollary}[section]

\newtheorem{example}{Example}[section]
\newtheorem{proposition}{Proposition}[section]
\newtheorem{lemma}{Lemma}[section]

\begin{document}

\title{Higher Analogues of Discrete Topological Complexity}

\author{H\.{I}lal Alabay\textsuperscript{1}, Ay\c{s}e Borat\textsuperscript{1,3}, Esra C\.{I}hang\.{I}rl\.{I}\textsuperscript{1}, Esma D\.{I}r\.{I}can Erdal\textsuperscript{2}}

\date{\today}

\address{\textsc{Hilal Alabay}
Bursa Technical University\\
Faculty of Engineering and Natural Sciences\\
Department of Mathematics\\
Bursa, Turkiye}
\email{hilalalabay@hotmail.com}

\address{\textsc{Ay\c{s}e Borat}
Bursa Technical University\\
Faculty of Engineering and Natural Sciences\\
Department of Mathematics\\
Bursa, Turkiye}
\email{ayse.borat@btu.edu.tr} 

\address{\textsc{Esra Cihangirli}
Bursa Technical University\\
Faculty of Engineering and Natural Sciences\\
Department of Mathematics\\
Bursa, Turkiye}
\email{esra.cihangirlii@gmail.com} 

\address{\textsc{Esma Dirican Erdal}
Istanbul Technical University\\
Faculty of Arts and Sciences\\
Department of Mathematics Engineering\\
Istanbul, Turkiye}

\address{
I\c{s}\i k University\\ 
Faculty of Engineering and Natural Sciences\\
Department of Mathematics\\
34980 Istanbul, Turkiye}
\email{ediricanerdal@itu.edu.tr} 
\email{esma.diricanerdal@isikun.edu.tr}



\begin{abstract} In this paper, we introduce the $n-$th discrete topological complexity and study its properties such as its relation with simplicial Lusternik-Schnirelmann category and how the higher dimensions of discrete topological complexity relate with each other. Moreover, we find a lower bound of $n-$discrete topological complexity which is given by the $n-$th usual topological complexity of the geometric realisation of that complex. Furthermore, we give an example for the strict case of that lower bound. 
\end{abstract}

\maketitle

\footnotetext[1]{The authors were supported by the Scientific and Technological Research Council of Turkiye (T\"{U}B\.{I}TAK) [grant number 122F295].}
\footnotetext[2]{The author was supported by \.{I}T\"{U}-DOSAP [grant number 43959].}
\footnotetext[3]{The corresponding author.}

\section{Introduction}

Topological complexity  is a topological invariant that was introduced by Farber in \cite{F2}. It is intended to solve problems such as motion planning in robotics. To catch this aim, a computable algorithm is needed, for each pair of points of 
the configuration space of a mechanical or physical device, a path connecting them in a continuous way. 
Farber used a well-known map in algebraic topology that is a section of the path-fibration and he interpreted that algorithm in terms of this map. In 2010, Rudyak introduced a notion of higher topological complexity in \cite{R}. After that a discrete version of topological complexity is established by Fernandez-Ternero, Macias-Virgos and Vilches in \cite{FMMV}. 

The importance of discretisation is based on the fact that many motion planning methods transform a continuous problem into a discrete one. So the aim of the present paper is to extend the Rudyak's approach to discrete version. Namely, we define the $n-$th discrete topological complexity.

Before we introduce the $n-$th discrete topological complexity, let us recall some known definitions and theorems. 


\begin{definition}
Let K be a simplicial complex. An edge path in K is a finite or infinite sequence of vertices such that any two consecutive vertices span an edge. We say that K is edge path connected if any two vertices can be joined by a finite edge path. 
\end{definition}

\begin{definition}
Two simplicial maps $\varphi, \psi : K \to L$ are said to be contiguous (denoted by $\varphi \sim_c \psi$) if for  every simplex $\{v_0, \dots, v_k\}$ in $K$, $\{\varphi(v_0), \dots , \varphi(v_k),\psi(v_0), \dots ,$\\
$\psi(v_k)\}$ constitutes a simplex in $L$.
    
\end{definition}

\begin{definition}
Two simplicial maps $\varphi, \psi : K \rightarrow L $ are said to be in the same contiguity class (denoted by $\varphi \sim \psi $) if there exists a finite sequence of simplicial maps $\varphi_i : K \rightarrow L$ for $i=0,1,\dots m$ such that $\varphi = \varphi_1 \sim_c \varphi_2 \sim_c \dots \sim_c \varphi_m= \psi$.
\end{definition}

For two simplicial complexes $K$ and $L$, if there is simplicial maps $\varphi: K\rightarrow L$ and $\psi: L\rightarrow K$ such that $\varphi\circ \psi \sim \id_K$ and  $\psi\circ \varphi \sim \id_{L}$, then $K$ and $L$ are said to have the same strong homotopy type, and is denoted by $K\sim L$. If $K\sim\{v_0\}$ where $v_0$ be a vertex in $K$, then $K$ is said to be strongly collapsible.

The cartesian product of simplicial complexes is not necessarily a simplicial complex. So a "new" product on simplicial complexes is defined to make their product into a simplicial complex. It is the categorical product of simplicial complexes and for two simplicial complexes $K_1$ and  $K_2$, it is defined as follows.

\begin{itemize}
\item[(i)] The set of vertices is defined by $V(K_1\prod K_2):=V(K_1)\times V(K_2)$.
\item[(ii)] If $p_i: V(K_1\prod K_2)\rightarrow V(K_i)$ is the projection map for $i=1,2$, then a simplex $\sigma$ is said to be in $K_1\prod K_2$ if $p_1(\sigma)\in K_1$ and $p_2(\sigma)\in K_2$.
\end{itemize} (for more details, see \cite{K}).

\begin{definition} \cite{FMV}
Let K be a simplicial complex. A subcomplex $\Omega\subset K$ is said to be categorical if the inclusion $i:\Omega\hookrightarrow K$ and a constant map $c_{v_0}: \Omega\rightarrow K$, where $v_0 \in K$ is some fixed vertex, are in the same contiguity class.
\end{definition}

\begin{definition} \cite{FMV} The simplicial Lusternik-Schnirelmann category $\scat(K)$ is the least integer $k\ge 0$ such that there exist categorical subcomplexes $\Omega_0,\Omega_1,\dots, \Omega_k$ of $K$ covering $K$.
\end{definition}

\begin{definition} \cite{FMMV} 
We say that $\Omega \subset K^2$ is a Farber subcomplex if there is a simplicial map $\sigma : \Omega \to K$ such that $\Delta\circ \sigma$ $\sim \iota_{\Omega}$ where $\Delta : K \to K^2$, $\Delta(v) = (v,v)$ is the diagonal map and $\iota_{\Omega} : \Omega \hookrightarrow K^2$ is the inclusion map.
\end{definition}

\begin{definition} \cite{FMMV} 
The discrete topological complexity $\TC(K)$ of a simplicial complex $K$ is the least non-negative integer $k$ such that $K^2$ can be covered by $k+1$ Farber subcomplexes. More precisely, $K^2 = \Omega_0 \cup \cdots \cup \Omega_k$ and there exist simplicial maps $\sigma_j : \Omega_j \to K$ satisfying $\Delta \circ \sigma_j \sim \iota_j$ where $\iota_j : \Omega_j \hookrightarrow K^2$ are inclusions for each $j = 0,\dots,k$. 
\end{definition}

\section{Higher Analogues of Discrete Topological Complexity}

\begin{proposition}
\label{prop 2.9}
If $\varphi, \psi : K \to L$ are in the same contiguity class, then so are $\varphi^n$ and $\psi^n$.
\end{proposition}

\begin{proof} Without loss of generality we will focus on being contiguous and show that if $\varphi \sim_c \psi$, so is $\varphi^n \sim_c \psi^n$. 

Suppose that $\varphi \sim_c \psi$. If 
\begin{center}
$\sigma = \{(a_{11},a_{12},\dots,a_{1m}),(a_{21},a_{22},\dots,a_{2m}),\dots,(a_{n1},a_{n2},\dots,a_{nm})\}$
\end{center}
is a simplex in $K^m$, then one can say that $\pi_1(\sigma)=\{a_{11},a_{21},\dots,a_{n1}\}$, $\pi_2(\sigma)=\{a_{12},a_{22},\dots,a_{n2}\}$, $\dots$, $\pi_n(\sigma)=\{a_{1m},a_{2m},\dots,a_{nm}\}$ are simplices of K. Hence, for each $i=1,2,\dots,m$,
\begin{center}
$\varphi(\pi_i(\sigma)) \cup \psi(\pi_i(\sigma)) = (\varphi(a_{1i}),\cdots,\varphi(a_{ni}),\psi((a_{1i}),\cdots,\psi(a_{ni})$
\end{center}
belongs to $L$, so that $\varphi^n(\sigma) \cup \psi^n(\sigma)$ belongs to $L^m$.
\end{proof}

\begin{definition}\label{def 2.2}
We say that $\Omega \subset K^n$ is an $n-$Farber subcomplex if there is a simplicial map $\sigma : \Omega \to K$ such that $\Delta\circ \sigma$ $\sim \iota_{\Omega}$ where $\Delta : K \to K^n$, $\Delta(v) = (v,v,\dots,v)$ is the diagonal map and $\iota_{\Omega} : \Omega \hookrightarrow K^n$ is the inclusion map.
\end{definition}

\begin{definition}
The $n-$th discrete topological complexity $\TC_n(K)$ of a simplicial complex $K$ is the least non-negative integer $k$ such that $K^n$ can be covered by $k+1$ $n-$Farber subcomplexes. More precisely, $K^n = \Omega_0 \cup \cdots \cup \Omega_k$ and there exist simplicial maps $\sigma_j : \Omega_j \to K$ satisfying $\Delta \circ \sigma_j \sim \iota_j$ where $\iota_j : \Omega_j \hookrightarrow K^n$ are inclusions for each $j = 0,\dots,k$. 
\end{definition}

\begin{theorem}
\label{teo 3.4}
For a subcomplex $\Omega \subset K^n$, the followings are equivalent.

\begin{itemize}
    \item[(1)] $\Omega$ is an n-Farber subcomplex.  
    \item[(2)] $(\pi_i)_|{}_\Omega \sim (\pi_j)_|{}_\Omega$ for all $i,j\in\{1,2,\ldots,n\}$. 
    \item[(3)] One of the restrictions $(\pi_1)_|{}_\Omega, (\pi_2)_|{}_\Omega, \ldots, (\pi_n)_|{}_\Omega $ is a section (up to contiguity) of the diagonal map $\Delta:K \rightarrow K^n$.
\end{itemize}   
\end{theorem}

\begin{proof}   

$(1) \Rightarrow (2)$: Suppose that $\Omega \subset K^n$ is an $n-$Farber subcomplex. Then there exists a simplicial map $\sigma : \Omega \rightarrow K $ such that $\Delta \circ \sigma \sim \iota_\Omega$. 

Observe that $\Delta\circ \sigma$ is an n-tuple of $\sigma$'s, that is, $\Delta\circ \sigma=(\sigma, \ldots, \sigma): \Omega \rightarrow K^n$ by $\Delta\circ \sigma(\omega)=(\sigma(\omega), \ldots, \sigma(\omega))$. 

On the other hand, the inclusion map $\iota_\Omega:\Omega\hookrightarrow K^n$ can be written as $$\iota_\Omega=(\pi_1|_{\Omega}, \ldots, \pi_n|_{\Omega}).$$ So we have 

\[
(\sigma, \ldots, \sigma)=\Delta\circ \sigma \sim (\pi_1|_{\Omega}, \ldots, \pi_n|_{\Omega}).
\]

\noindent Hence, $\sigma\sim \pi_i|_{\Omega}$ for each $i$.

$(2) \Rightarrow (3)$: Fix $i_0 \in \{1, 2, \ldots, n\}$ and suppose that $\pi_i|_\Omega \sim \pi_{i_0}|_\Omega$ for all $i\in\{1,2,\ldots,n\}$. Then 

\[
\iota_\Omega=\big( \pi_1|_\Omega , \ldots, \pi_n|_\Omega \big)\sim \big( \pi_{i_0}|_\Omega , \ldots, \pi_{i_0}|_\Omega \big) = \Delta \circ \pi_{i_0}|_\Omega.
\]

\noindent which means that one of the $\pi_{i_0}|_\Omega$'s is a section of $\Delta$.

$(3) \Rightarrow (1)$: Suppose that $(\pi_i{})_|{}_\Omega$ is a section (up to contiguity) of the diagonal map $\Delta$, for some $i\in \{ 1,2,\ldots,n\}$ and choose the simplicial map $\sigma:=(\pi_i{})_|{}_\Omega :\Omega\rightarrow K$. Then $\Omega$ is an $n-$Farber subcomplex.
\end{proof}

\begin{theorem}\label{dimension} $\TC_m(K)\leq \TC_{m+1}(K)$.
\end{theorem}

\begin{proof}
Let $\TC_{m+1}(K)=\ell$. Then there exist $\Omega_0, \ldots, \Omega_{\ell}$ $(m+1)-$Farber subcomplexes of $K^{m+1}$ covering $K^{m+1}$. By Theorem~\ref{teo 3.4}, 

\[
(\pi_1)_|{}_{\Omega_j} \sim (\pi_2)_|{}_{\Omega_j} \sim \ldots \sim (\pi_{m+1})_|{}_{\Omega_j}
\]

\noindent for each $j\in\{0,1,\ldots,\ell\}$, where $\pi_i:K^{m+1}\rightarrow K$ is the projection to the $i$-th factor.

If we show that there are $\Lambda_0, \ldots, \Lambda_{\ell}$ subcomplexes of $K^{m}$ covering $K^{m}$ such that $(\pi^1)_|{}_{\Lambda_j} \sim (\pi^2)_|{}_{\Lambda_j} \sim \ldots \sim (\pi^{m})_|{}_{\Lambda_j}$ for each $j\in\{0,1,\ldots,\ell\}$, where $\pi^i:K^m\rightarrow K$ is the projection to the $i$-th factor, then we are done. 

For a fixed vertex $\omega \in K$, define the simplicial map 

\[
g: K^m\rightarrow K^{m+1} , \hspace{0.1in} \text{by} \hspace{0.1in} g(v^1,\ldots, v^m)=(v^1,\ldots, v^m,\omega)
\]

\noindent and for each $j\in\{0,1,\ldots,\ell\}$, define the subcomplex $\Lambda_j := g^{-1}(\Omega_j)$ .

We have $\pi^i |_{\Lambda_j}=\pi_i |_{\Omega_j}\circ g : \Lambda_j \rightarrow \Omega_j \rightarrow K$, for $i\in\{1,\ldots,m\}$ and  $j\in\{0,1,\ldots,\ell\}$.

For every $j\in\{0,1,\ldots,\ell\}$, since $(\pi_i)_|{}_{\Omega_j} \sim (\pi_{k})_|{}_{\Omega_j}$ for each $i,k\in\{1,\ldots, m+1\}$, we have $(\pi^i)_|{}_{\Lambda_j} \sim (\pi^{k})_|{}_{\Lambda_j}$ for each $i,k\in\{1,\ldots, m\}$.
\end{proof}

\begin{theorem}
If $K \sim L$, then $\TC_n(K) = \TC_n(L)$, i.e., the $n-$th discrete topological complexity is an invariant of the strong homotopy type. 
\end{theorem}

\begin{proof}
We proceed in three steps. 

\textit{Step 1.} Let us show that if $K\sim L$, then $K^n \sim L^n$. \\
Suppose that $K$ and $L$ are of the same strong homotopy type, then by definition, there are simplicial maps $\varphi, \psi$ satisfying $\varphi\circ \psi\sim 1_L$ and $\psi\circ \varphi\sim 1_K$. First consider the case $\varphi\circ \psi\sim 1_L$. By Proposition~\ref{prop 2.9}, we obtain

\[\varphi^n \circ \psi^n =(\varphi\circ \psi)^n \sim (1_L)^n = 1_{L^n}.
\]

Using same arguments, we show that $\psi^n \circ \varphi^n \sim 1_{K^n}$. Hence, we have $K^n \sim L^n$. 

\textit{Step 2.} Next we show that if $\Omega \subset K^n$ is an $n-$th Farber subcomplex, then so is $(\psi^n)^{-1}(\Omega)\subset L^n$.

There exists a simplicial map $\sigma : \Omega \to K$ so that $\Delta_K \circ \sigma \sim \iota_{\Omega}$ is satisfied, by definition, provided that $\Omega \subset K^n$ is an $n-$th Farber subcomplex. 

On the other hand, from the following commutative diagram, it follows that $\Delta_L \circ \varphi = \varphi^n \circ \Delta_K$ and $\Delta_K \circ \psi = \psi^n \circ \Delta_L$.

\[
\begin{tikzcd}
	& K \arrow[r, "\varphi"] \arrow[d, "\Delta_K"] & \arrow[l, shift left, "\psi"] L  \arrow[d, "\Delta_L"] \\
    & K^n \arrow[r, shift left, "\varphi^n"]  & \arrow[l, shift left, "\psi^n"] L^n 
\end{tikzcd}
\]

Combining these two facts, we have the diagram 

	\[
	\begin{tikzcd}
	& K \arrow[r, "\varphi"] \arrow[d, "\Delta_K"] & \arrow[l, shift left, "\psi"] L  \arrow[d, "\Delta_L"]  \\
	\Omega \arrow[r, "\iota_\Omega"]   \arrow{ur}{\sigma}  & K^n \arrow[r, shift left, "\varphi^n"]  & \arrow[l, shift left, "\psi^n"] L^n 
	& \arrow[l, shift left, "\iota_{\widetilde{\Omega}}"] \widetilde{\Omega} \arrow{ul}{}
	\end{tikzcd}
	\]

\noindent and we find out that $\widetilde{\Omega}:=(\psi^n)^{-1}(\Omega)\subset L^n$ is an $n-$Farber subcomplex as there exists a simplicial map 

\[\widetilde{\sigma} := \varphi\circ\sigma\circ (\varphi^n)|_{\widetilde{\Omega}}: \widetilde{\Omega}\rightarrow L
\]
 
\noindent satisfing the following

\begin{center}
\begin{eqnarray*}
\Delta_L \circ \widetilde{\sigma} &=& \Delta_L \circ \varphi \circ \sigma \circ \psi^n |_{\widetilde{\Omega}} \\
&=& \Delta_L \circ \varphi \circ \sigma \circ \psi^n \circ \iota_{\widetilde{\Omega}} \\
&=& \varphi^n \circ \Delta_K \circ \sigma \circ \psi^n \circ \iota_{\widetilde{\Omega}}\\
&\sim & \varphi^n \circ \iota_\Omega \circ \psi^n \circ \iota_{\widetilde{\Omega}}\\
&=& (\varphi^n \circ \psi^n)_{|\widetilde{\Omega}}\\
&\sim & 1_{L^n} \circ \iota_{\widetilde{\Omega}}\\
&=& \iota_{\widetilde{\Omega}} \hspace{0.03in}.
\end{eqnarray*}
\end{center}

\textit{Step 3.} In the last step, we show that $\TC_n(K) \leq \TC_n(L)$. Similarly,  $\TC_n(K) \geq \TC_n(L)$ can be showed and the result follows.

Let say $\TC_n(K) = k$. So there is a covering $K^n = \Omega_0 \cup \cdots \cup \Omega_k$ such that each $\Omega_j$ is an $n-$Farber subcomplex. From Step 2, each $\Lambda_j = (\psi^n)^{-1} (\Omega_j)$ is an $n-$Farber subcomplex for $j\in\{ 0, \dots ,k\}$ and they cover $L^n$. Hence, $\TC_n(K) \leq k$. 
\end{proof}

\begin{proposition}\cite[Lemma 4.4]{FMMV}\label{4.4} $K$ is path-edge connected if and only if any two constant simplicial maps to $K$ are in the same contiguity class.
\end{proposition}

\begin{theorem}\label{scatTC} $\scat(K^{n-1}) \leq \TC_n(K)$ provided that $K$ is an edge-path connected simplicial complex $K$.
\end{theorem}

\begin{proof} Suppose $\TC_n(K)=k$. Then there exist subcomplexes $\Omega_0, \Omega_1, \ldots, \Omega_k$ of $K^n$ covering $K^n$ such that each $\Omega_i$ is $n-$Farber subcomplex, i.e., there exists $\sigma_i: \Omega_i\rightarrow K$ satisfying $\Delta_K \circ \sigma_i \sim \iota_{\Omega_i}$, where $\iota_{\Omega_i}: \Omega_i \hookrightarrow K^n$ is the inclusion and $\Delta_K: K\rightarrow K^n$ is the diagonal map.

For a fixed vertex $v_1^0$, define $\iota_0: K^{n-1}\rightarrow K^n$ by $\iota_0(\omega)=(v_1^0,\omega)$. Set $V_i:=\iota_0^{-1}(\Omega_i)\subset K^{n-1}$ subcomplex. Notice that $V_0, V_1, \ldots, V_k$ cover $K^{n-1}$. If we prove that each $V_i$ a categorical subcomplex of $K^{n-1}$, then we can conclude that $\scat(K^{n-1})\leq k$, and the result follows.

For simplicity, we ignore the subscripts $i$. From the assumption, $\Delta_K \circ \sigma \sim \iota_{\Omega}$, that is, there exist $\varphi_{\bar{j}}: \Omega\rightarrow K^n$ simplicial maps, for ${\bar{j}}\in\{1, 2, \ldots, m\}$, such that $\varphi_1=\Delta_K\circ \sigma$, $\varphi_m=\iota_{\Omega}$ and $\varphi_{\bar{j}}$ is contiguous to $\varphi_{{\bar{j}}+1}:$

\[
\Delta_K \circ \sigma = \varphi_1 \sim_c \varphi_2 \sim_c \ldots \sim_c \varphi_m = \iota_{\Omega}.
\]

Denoting by $\pr_j: K^n \rightarrow K$ the projection to the $j$-th factor, take the compositions

\begin{eqnarray*}
\pr_j\circ (\Delta_K \circ \sigma) \circ \iota_0 \circ \iota_{V}&=&\pr_j\circ \varphi_1 \circ \iota_0 \circ \iota_{V}\\
&\sim_c & \pr_j\circ \varphi_2 \circ \iota_0 \circ \iota_{V} \\
&\sim_c & \ldots \\
&\sim_c &  \pr_j\circ \varphi_m \circ \iota_0 \circ \iota_{V} \\
&=& \pr_j\circ \iota_{\Omega}\circ \iota_0 \circ \iota_{V}
\end{eqnarray*}

\noindent for each $j=1,\ldots, n$. 

Notice that $\pr_j\circ \Delta_K \circ \sigma \circ \iota_0 \circ \iota_{V}(\omega)= \sigma(v_1^0,\omega)\in K$ for each $j\in\{1,\ldots,n\}$.

On the other hand 

\[
\pr_1\circ (\Delta_K \circ \sigma)\circ \iota_0 \circ \iota_{V} = v_1^0 
\]
\[
\pr_j\circ (\Delta_K \circ \sigma)\circ \iota_0 \circ \iota_{V} = \pi_{j-1} 
\]

\noindent for $j\in\{2,\ldots,n\}$, where $\pi_j:K^{n-1}\rightarrow K$ is the projection to the $j$-th factor. Here, we can write $\iota_{V}$ in terms of $\pi_j$'s. Thus,

\[
\iota_{V}=\big( \pi_1,\pi_2,\ldots, \pi_{n-1} \big) \sim \big( \pr_2\circ (\Delta_K\circ\sigma)\circ \iota_0\circ \iota_{V}, \ldots, \pr_n\circ (\Delta_K\circ\sigma)\circ \iota_0\circ \iota_{V} \big).
\]

Since $K$ is path-edge connected, by Lemma~\ref{4.4}, the constant map $c:V\rightarrow K^{n-1}$, $c(\omega)=(v_1^0, \ldots, v_{n-1}^0)$ can be realised as another constant map $\bar{c}: V\rightarrow K^{n-1}$, $\bar{c}(\omega)=(v_1^0, \ldots, v_1^0)$. Therefore, $\iota_{V}\sim c$.
\end{proof}

\begin{theorem}\label{2.4} $\TC_n(K) \leq \scat(K^n)$ provided that $K$ is an edge-path connected simplicial complex $K$.
\end{theorem}

\begin{proof}
Suppose that $\scat(K^n) = k$. Then there is a categorical covering $\{U_0,U_1,\ldots, U_k\}$ of $K^n$. If we show that $U_i$ is an $n-$Farber subcomplex, for each $i\in\{0,1,\ldots,k\}$, then the proof is concluded.


Since we have $\iota\sim c$ where $c:U\rightarrow K^n$ is a constant map and $\iota: U \rightarrow K^n$ is the inclusion map, there exists a sequence of simplicial maps $h_t: V \rightarrow K^n$ for $t\in\{1,\ldots,m\}$ such that $h_0=\iota$, $h_m=c$ and  $(h_{t},h_{t+1})$ are contiguous for all $t\in\{1,2,\ldots,m-1\}$. 

Now let $\pi_j:K^n\rightarrow K$ denote the projection map to the $j$-th factor. Hence, $\pi_j\circ h_{t}$ and $\pi_j \circ h_{t+1}$ are contiguous for all  $t\in\{1,2,\ldots,m-1\}$. From the fact that

\[
\pi_j\circ\iota=\pi_j\circ h_0 \sim \pi_j\circ c,  
\]

\noindent it follows that $\pi_j\circ\iota$'s are all in the same contiguity class for all $j$. By Theorem~\ref{teo 3.4}, $U$ is an $n-$Farber subcomplex.
\end{proof}

The following lemma, which can be proved by induction, is a generalisation of Theorem 5.5 in \cite{FMMV2} and it will be used to prove the later corollary. 

\begin{lemma}\label{lemma2.1} For finite simplicial complexes $K_1, K_2, \ldots, K_m$, we have 
\[
\scat(K_1 \times K_2 \times \ldots \times K_m)+1 \leq (\scat K_1 +1)(\scat K_2 +1)\ldots (\scat K_m +1).
\]
\end{lemma}

\begin{corollary}\label{c}
Let $K$ be an abstract simplicial complex. Then $K$ is strongly collapsible if and only if $\TC_n(K)=0$.
\end{corollary}

\begin{proof}

We have $\scat{K}=0$, since K is strongly collapsible. On the other hand, by Lemma~\ref{lemma2.1}, we obtain $\scat(K^n)+1 \leq (\scat{K}+1)^n$. Combining these two, we have $\scat(K^n)+1 \leq 1$. Hence, $\scat(K^n)=0$. It follows from Theorem~\ref{2.4} that $\TC_n(K)=0$.

On the other hand, if $\TC_n(K)=0$, then by Theorem~\ref{dimension} $\TC_2(K)\leq \TC_n(K)=0$. Hence the result follows from Corollary 4.7 in \cite{FMMV}.
\end{proof}




\subsection{Geometric Realisation}

It is proved that $|K^2|\approx |K|^2$ and $|K\times K|\sim |K|^2$, see Theorem 10.21 and  Proposition 15.23 in \cite{K}. Moreover, the higher dimensional versions are also valid, as mention in Remark 5.2 in \cite{FMMV2}. Combining these facts with Lemma 5.1 in \cite{FMMV}, we get the following lemma.

\begin{lemma}\label{lemma2.6} There is a homotopy equivalence $u: |K|^n \rightarrow |K^n|$ such that the following diagram is commutative for each $i\in\{1,\ldots,n\}$

\[
\begin{tikzcd}
{|K|^n} \arrow{r}{u} \arrow[swap]{dr}{p_i} & {|K^n|} \arrow{d}{|\pi_i|} \\
    & {|K|}  
\end{tikzcd}
\]
	
\noindent where $p_i: |K|^n \rightarrow |K|$ and $\pi_i:K^n\rightarrow K$ are projections.

\end{lemma}

\begin{theorem}\label{theorem2.7} $\TC_n(|K|) \leqslant \TC_n(K)$.

\end{theorem}

\begin{proof} Let $\TC_n(K) = k$. So there exist $n-$Farber subcomplexes $\Omega_0, \dots \Omega_k$ of $K^n$ covering $K^n$. By Theorem~\ref{lemma2.1},  $\pi_i \circ \iota_{\Omega_\ell}$ and $\pi_j \circ \iota_{\Omega_\ell}$ are in the same contiguity class for each pair $i,j\in\{1,2,\ldots, n\}$ and each $\ell\in\{0,1,\ldots,k\}$. Considering the geometric realisations of these simplicial maps, we have

\[
 |\pi_i \circ \iota_{\Omega_\ell}| \backsimeq |\pi_j \circ \iota_{\Omega_\ell}|
\]

since being in the same contiguity class can be realised as being homotopic in the continuous realm (see for details \cite{S}). 

On the other hand, $ |\pi_i \circ \iota_{\Omega_\ell}| =  |\pi_i|\circ |\iota_{\Omega_\ell}|$ and $|\iota_{\Omega_\ell}|=\iota_{|\Omega_\ell|}$. Hence, we get  

\[
|\pi_i|\circ \iota_{|\Omega_\ell |} \backsimeq |\pi_j| \circ \iota_{|\Omega_\ell |}. 
\]

Now consider the preimage $F_\ell = u^{-1}(\Omega_\ell) \subset |K|^n$ for each $\ell=0,1,\ldots,k$ where $u$ is a homotopy equivalence as given in Lemma~\ref{lemma2.6}. Here, all $F_\ell$'s are closed and the subsets $F_0, F_1, \ldots, F_k$ cover $|K|^n$. Moreover, 

\[
p_i \circ \iota_{F_\ell} = |\pi_i|\circ u \circ \iota_{F_\ell} = |\pi_i| \circ \iota_{|\Omega_\ell |} \backsimeq |\pi_j| \circ \iota_{|\Omega_\ell |} = |\pi_j|\circ u \circ \iota_{F_\ell} = p_j \circ \iota_{F_\ell}.
\]

\noindent Here, the first and the last equalities follow from Lemma~\ref{lemma2.6}. This completes the proof.
\end{proof}

The following is an example for the strict case of the inequality in Theorem~\ref{theorem2.7}.

\begin{example} Consider the simplicial complex in Figure 1. 

\[
\begin{tikzpicture}[scale=0.5]
 
\draw [fill=gray, gray] (0,0) -- (0,0) -- (4,6.8) -- (8,0);
\draw (0,0) -- (3,2.8);
\draw (0,0) -- (4,1);
\draw (4,6.8) -- (3,2.8);
\draw (4,6.8) -- (5,2.8);
\draw (8,0) -- (4,1);
\draw (8,0) -- (5,2.8);
\draw (3,2.8) -- (4,1);
\draw (3,2.8) -- (5,2.8);
\draw (5,2.8) -- (4,1);
\draw (0,0) -- (4,6.8);
\draw (0,0) -- (8,0);
\draw (8,0) -- (4,6.8);
\end{tikzpicture}
\]
\[
\textit{Figure 1}
\]

As mentioned in \cite{BM}, $K$ is not strongly collapsible. By Example 3.3 in \cite{FMV}, $\scat(K)=1$. Theorem~\ref{scatTC} yields that $\scat(K)\leq \TC_2(K)$. From Theorem~\ref{dimension}, it follows that $1\leq\TC_n(K)$ for all $n\geq 2$. On the other hand, the geometric realisation of the simplicial complex in Figure 1 is homeomorphic to a disc. Thus, $\TC_n(|K|)=0$ for any $n\geq 2$.
\end{example}

\end{document}